\documentclass[12pt]{amsart}
\usepackage{mathpazo}
\usepackage{amssymb}
\usepackage{csquotes}
\usepackage{hyperref}
\hypersetup{
colorlinks=true,
linkcolor=blue,
filecolor=blue,
urlcolor=blue,
citecolor=blue,
pdfpagemode=Fullscreen
}
\newtheorem{theorem}{Theorem}[section]
\newtheorem{lemma}[theorem]{Lemma}
\newtheorem*{Acknowledgement}{\textnormal{\textbf{Acknowledgement}}}
\theoremstyle{definition}
\newtheorem{definition}[theorem]{Definition}
\newtheorem{example}[theorem]{Example}

\newtheorem{corollary}[theorem]{Corollary}
\newtheorem{proposition}[theorem]{Proposition}

\newtheorem{remark}[theorem]{Remark}

\numberwithin{equation}{section}
\newcommand{\beqa}{\begin{eqnarray*}}
\newcommand{\eeqa}{\end{eqnarray*}}
\newcommand{\beqn}{\begin{eqnarray}}
\newcommand{\eeqn}{\end{eqnarray}}

\renewcommand{\a}{\alpha}

\newcommand{\e}{\varepsilon}

\newcommand{\la}{\lambda}

\newcounter{cnt1}
\newcounter{cnt2}
\newcounter{cnt3}
\newcommand{\blr}{\begin{list}{$($\roman{cnt1}$)$}
        {\usecounter{cnt1} \setlength{\topsep}{0pt}
                \setlength{\itemsep}{0pt}}}
\newcommand{\bla}{\begin{list}{$($\alph{cnt2}$)$}
        {\usecounter{cnt2} \setlength{\topsep}{0pt}
                \setlength{\itemsep}{0pt}}}
\newcommand{\bln}{\begin{list}{$($\arabic{cnt3}$)$}
        {\usecounter{cnt3} \setlength{\topsep}{0pt}
                \setlength{\itemsep}{0pt}}}
\newcommand{\el}{\end{list}}
\newtheorem{thm}{Theorem}

\newtheorem{Def}[thm]{Definition}

\newtheorem{rem}[thm]{Remark}
\newcommand{\Rem}{\begin{rem} \rm}
\newcommand{\bdfn}{\begin{Def} \rm}
\newcommand{\edfn}{\end{Def}}

\title[ ]{ball separation characterization of ball dentability and  related properties}
\author[ S. Basu, \ S. Seal ]
			{Sudeshna Basu$^{1}$, Susmita Seal$ ^{2}$ }
		\address{{$^{1}$}   Sudeshna Basu,
		Department of Mathematics and Statistics, 
				Loyola University, 
				Baltimore, MD 21210, USA 
				}
		\email{sudeshnamelody@gmail.com}
			
		\address {{$^{2}$} Susmita Seal, 
				School of Mathematical Sciences, National Institute of Science Educational and Research Bhubaneswar, An OCC of Homi Bhabha National Institute, P.O. - Jatni, District - Khurda, Odisha - 752050, India}
		\email{susmitaseal1996@gmail.com, susmitasealmath@niser.ac.in}

			\date{}
\begin{document}
\maketitle
\begin{abstract}
	In Euclidean spaces, every closed, bounded, convex set can be characterized by two equivalent notions of separation properties. This is not true in general for arbitrary Banach spaces. In this work, we present  a  ball separation characterization for spaces where the unit ball is dentable. We also explore related properties.


\end{abstract}

\section{Introduction}
The classical notion of separation in Functional Analysis ensures that a point outside a closed bounded convex set can always be separated by a hyperplane.
The ball separation property, considers whether a point outside a closed bounded convex set can be separated using closed balls rather than half-spaces.
The two  notions are equivalent  in finite-dimensional normed linear spaces  but  its equivalence in  general Banach spaces is highly nontrivial and uncovers rich duality relationships.
The study of ball separation in Banach spaces originated with Mazur \cite{M}. A Banach space $X$ has the \emph{Mazur Intersection Property} (MIP) if any closed, bounded, convex set in $X$ is the intersection of closed balls. Equivalently $X$ has the MIP if and only if all points on the dual unit sphere are semi denting points (see \cite{CL3}).  The notion of ball separation was further developed by Phelps \cite{P}, who established a duality between ball separation in a Banach space $X$ and the dentability of the closed unit ball in its dual space $X^*$. In a foundational result, Giles, Gregory, and Sims \cite{GGS} showed that the set of $w^*$-denting points of the dual unit ball  is norm dense in the dual unit sphere if and only if $X$ has the MIP.
Later, Chen and Lin \cite{CL1} obtained ball separation characterization of some of the well known geometric properties of Banach spaces. They demonstrated that the existence of even a single $w^*$-denting or $w^*$-Point of Continuity ($w^*$-PC) point in $B_{X^*}$ 
ensures a version of the ball separation property weaker than MIP.
They also extended the idea of MIP and introduced a refined geometric property known as \emph{Property (II)} : every closed bounded convex set is an intersection of closed convex hulls of finitely many closed balls. Moreover, they proved that $X$ satisfies Property (II) if and only if the set of $w^*$-PC of $B_{X^*}$ is norm dense in $S_{X^*}$.



Despite substantial progress in understanding ball separation and its connections to dentability and related geometric properties, a comprehensive treatment of ball separation characterizations of  a large class of Banach spaces is still unknown. In particular the ball separation charactersisation of the class of Banach spaces with dentable  closed unit ball remains to be explored.  Likewise, pointwise variants of these properties have yet to be systematically examined within the ball separation framework.


In this work, we give a  ball separation characterization of spaces where the unit ball is dentable. We extend these  ball separation characterization to a broader class of geometric properties in Banach spaces. Our main contributions are as follows.
In Section 3, we provide ball separation characterizations for Banach spaces with   small diameter properties  namely when the unit ball or the dual unit ball is dentable.
       In Section 4, motivated by the concept of semi denting point \cite{CL2}, we introduce and study pointwise versions of  small diameter properties : semi Point of Continuity (semi PC) and semi Small Combination of Slice (semi SCS) points. We give ball separation characterizations of all these properties and their $w^*$-versions.  We show that a Banach space $X$ has Property (II)  \cite{CL} if and only if every point of $S_{X^*}$ is a semi $w^*$-PC point of $B_{X^*}$.
         In Section 5, following \cite{CL},  we define the notion of an $\mathcal{A}$-Small Combination of Slice point ($\mathcal{A}$-SCS point) in $B_{X^*}$, which generalizes the concept of $w^*$-SCS points. We establish a necessary ball separation condition for the existence of such points. Furthermore, we prove that if the linear span of $\mathcal{A}$-SCS points is dense in $X^*$ (with respect to a topology $\tau_\mathcal{A}$), then every closed, bounded, convex subset in a compatible collection $\mathcal{A} \subset X$ is ball generated, extending a result from \cite{BB}.
  In some proofs, we adapt techniques from the previous works
of Chen and Lin \cite{CL1} and Giles \cite{G}, leading to several new results and unifying existing ones within a common geometric framework.

\section{Notations and Preliminaries}

Throughout this work, we consider only real Banach spaces. 
We  denote the closed unit ball of a Banach space $X$ by $B_X$, the unit sphere by $S_X$ and the closed ball of radius $r >0$ and center $x$ by $B(x, r).$ 
 For any subset $C$ of $X$ we denote its closure by $\overline{C}$, its convex hull by co($C$), its absolute convex hull by $|\mathrm{co}|$($C$) and its diameter by diam$(C)$. Given two subsets 
$A$ and $B$ of $X$, the usual distance between them is denoted by $d(A,B)$.
 For any bounded $K\subset X$,  we define
$\|x^*\|_K = \sup\{|x^* (x)| : x \in K\}$, $x^*\in X^*$, which induces 
 a seminorm on $X^*$. 
The diameter of $C\subset X^*$ under this seminorm is given by $\mathrm{diam}_K C=\sup\{\|x^*- y^*\|_K : x^*, y^*\in C\}.$

A \emph{slice} of a bounded set $C\subset X$ is defined by 
$$S(C, x^*, \a) = \{x \in C : x^*(x) > \sup\limits_{c\in C} x^*(c) - \a \}$$
where $x^*\in X^*$ and $\a > 0.$
 We assume without loss of generality that $\|x^*\| = 1$. 
 Analogously one can define $w^*$-slices in $X^*$ by considering the determining functional from the predual space instead of dual space.
The term convex combination of slices  of $C$ refers to expressions of the form $\sum\limits_{i=1}^{n} \lambda_i S_i$
where $\lambda_1,\ldots,\lambda_n>0$ with $\sum\limits_{i=1}^{n} \lambda_i =1,$ and $S_1, \ldots, S_n$ are slices of $C$.
Similarly,
one can define $w^*$-convex combination of slices in $X^*$. 
\begin{definition}
A  point $ x \in S_{X}$ is called a \emph{denting point} (resp. \emph{Point of Continuity} (PC), \emph{Small Combination of Slice point} (SCS point)) of $B_X$ if for any $\varepsilon> 0,$  there exists a slice (resp. weakly open set, convex combination of  slices)  $S$ of $B_{X}$ such that $x \in S \subset B(x,\varepsilon) $.
\end{definition}
Analogously the notions of $w^*$-denting , $w^{*}$-PC and $w^{*}$-SCS point  in dual space is defined by considering $w^*$-slices, $w^*$-open sets and convex combination of $w^*$-slices respectively.

	
		

		

\begin{definition}
	\cite{BS} A Banach space $X$ has 
	
	\begin{enumerate}
		
		\item  \emph{Ball Dentable Property} ($BDP$) if $B_X$ has  
		slices of arbitrarily small diameter. 
		
		\item  \emph{Ball Huskable Property} ($BHP$) if $B_X$ 
		has nonempty relatively weakly open subsets of arbitrarily small diameter.
		
		\item  \emph{Ball Small Combination of Slice Property} ($BSCSP$) 
		if  $B_X$ has convex combination of slices of arbitrarily 
		small diameter.
	\end{enumerate}
	
\end{definition}
Analogously we can define $w^*$-$BDP$ , $w^{*}$-$BHP$ and $w^{*}$-$BSCSP$ point  in dual space by considering $w^*$-slices, $w^*$-open sets and convex combination of $w^*$-slices respectively. These properties, collectively referred to as  small diameter properties, have significant geometric implications. For further details, see
\cite{B},\cite{BR},\cite{BS},\cite {BS1}  and \cite{BGSY}.

\section{Ball separation characterization of Small Diameter Properties}

We begin by recalling the following results essential for our subsequent discussions.

\begin{proposition} \label{A1}
	\cite{BS} A Banach space $X$ has $BDP$ (respectively, $BHP$, $BSCSP$) if and only if $X^{**}$ has $w^*$-$BDP$ (respectively, $w^*$-$BHP$, $w^*$-$BSCSP$). 
\end{proposition}

\begin{lemma}\label{phelp's lemma}
\cite{CL1} 
For a normed space $X,$ let $f$ and $g$ be elements in $S_{X^*},$ and let $A=\{x\in B_X:f(x)>\frac{\varepsilon}{2}\}$ where $0<\varepsilon<1.$ If $\inf g(A)>0,$ then $\|f-g\|<\varepsilon.$
\end{lemma}
\begin{lemma}\label{cl lemma}
\cite[Lemma 1.1]{CL1} 
Let 
$x\in S_X$, $\delta>0$ and $\varepsilon>0.$ If 
diam $( S(B_{X^*},x,\delta))\leqslant \varepsilon,$
 then 
$$\sup_{y\in B_X} \frac{\|x+\frac{\delta}{2} y\|+\|x-\frac{\delta}{2} y\|-2}{\frac{\delta}{2}}\leqslant \varepsilon$$
\end{lemma}

\begin{theorem} \label{ch w star bdp}
For a Banach space $X,$ the following are equivalent :
\begin{enumerate}
\item $X^*$ has $w^*$-$BDP.$
\item Given $\varepsilon>0,$ there exists $x^*_0\in S_{X^*}$ such that
for every subset $C$ of $B_X$ with $\inf x^*_0(C)> \varepsilon,$ there exists a closed ball $B$ in $X$ satisfying $C\subset B$ and $\inf x^*_0(B)>0$.
\item Given $\varepsilon>0,$ there exists $x^*_0\in S_{X^*}$ such that 
for every subset $C$ of $B_{X^{**}}$ with $\inf x^*_0(C)> \varepsilon,$ there exists a closed ball $B$ in $X^{**}$ with center in $X$ satisfying $C\subset B$ and $\inf x^*_0(B)>0$.
\end{enumerate}
\end{theorem}

\begin{proof}
 (i) $\Rightarrow$ (ii). Let $\varepsilon>0.$ Since $X^*$ has $w^*$-$BDP,$ then there exists a $w^*$-slice $S(B_{X^*},\tilde{x},2\alpha)$ of $B_{X^*}$ with diameter less than $\frac{\varepsilon}{2}.$ Choose $x^*_0\in S_{X^*}$ such that $x_0^*(\tilde{x})=1$.
 We aim to show that  $x^*_0 $  satisfies the condition in (ii).  

Consider a  subset $C\subset B_X$ with 
$\inf x^*_0(C)> \varepsilon.$ 
 By Lemma $\ref{cl lemma},$
\begin{equation}\label{1 chap 6}
\sup_{y\in B_X} \frac{\|\tilde{x}+\alpha y\|+\|\tilde{x}-\alpha y\|-2}{\alpha}\leqslant \frac{\varepsilon}{2}.
\end{equation}
 We claim that  $C\subset B(\frac{\tilde{x}}{\alpha},\frac{1}{\alpha}-\frac{\varepsilon}{2}).$ \\
 If not, there exists $y_0\in C$ such that $\|\frac{\tilde{x}}{\alpha} - y_0\| > \frac{1}{\alpha}-\frac{\varepsilon}{2}.$
 Then
 $$1+\alpha x_0^*(y_0)=x_0^* (\tilde{x}+\alpha y_0)\leqslant \|\tilde{x}+\alpha y_0\|$$
 Therefore, $x_0^*(y_0)\leqslant \frac{\|\tilde{x}+\alpha y_0\|-1}{\alpha}$.
 
  Note that, 
  \begin{equation} \notag
  \begin{split}
  \frac{\|\tilde{x}+\alpha y_0\|+\|\tilde{x}-\alpha y_0\|-2}{\alpha}
 =\frac{\|\tilde{x}+\alpha y_0\|-1}{\alpha} + \|\frac{\tilde{x}}{\alpha} - y_0\| - \frac{1}{\alpha} \hspace{1 cm}\\
 > x_0^* (y_0) -\frac{\varepsilon}{2}\hspace{4 cm}\\
 \geqslant \inf x_0^*(C) -\frac{\varepsilon}{2}\hspace{3.5 cm}\\
 >\varepsilon -  \frac{\varepsilon}{2}
 =\frac{\varepsilon}{2}\hspace{3.9 cm}
 \end{split}
  \end{equation}
 This contradicts $(\ref{1 chap 6}).$ Hence our claim holds.
 
  Also, $\inf x_0^*(B(\frac{\tilde{x}}{\alpha},\frac{1}{\alpha}-\frac{\varepsilon}{2}))=x_0^*(\frac{\tilde{x}}{\alpha})- (\frac{1}{\alpha}-\frac{\varepsilon}{2})=\frac{1}{\alpha}- (\frac{1}{\alpha}-\frac{\varepsilon}{2})=\frac{\varepsilon}{2} >0.$
  

(ii) $\Rightarrow$ (i). 
Let $\varepsilon>0.$ Then there exists $x^*_0\in S_{X^*}$ satisfying the condition in (ii). Choose $0<\delta<\frac{\varepsilon}{2}$. Define 
$$K_{\delta}=\{x\in B_X: x^*_0(x)>\varepsilon+ \delta\}.$$
 Then $\inf x^*_0(K_{\delta})>\varepsilon.$ Therefore by (ii), we can choose a closed ball $B=B(\tilde{x},\tilde{r})$ in $X$ such that $K_{\delta}\subset B$ and $\inf x^*_0(B)>0$.
 Choose $0<\beta <\min \{\varepsilon - 2\delta , 1-\frac{\tilde{r}}{\|\tilde{x}\|}\}$.
Consider the $w^*$-slice $S=S(B_{X^*},\frac{\tilde{x}}{\|\tilde{x}\|}, \beta)$ in $B_{X^*}.$
Let $x^*\in S$. Then $\frac{x^*}{\|x^*\|} \in S\bigcap S_{X^*}.$ Thus
\begin{equation}
\frac{x^*}{\|x^*\|}(\tilde{x})=\|\tilde{x}\| \ \frac{x^*}{\|x^*\|} \Big(\frac{\tilde{x}}{\|\tilde{x}\|}\Big)> \|\tilde{x}\| \ (1-\beta) \geqslant \|\tilde{x}\| \ [1-(1-\frac{\tilde{r}}{\|\tilde{x}\|})]=\tilde{r}
\end{equation}
and therefore
 \begin{equation}
\inf \frac{x^*}{\|x^*\|}(K_{\delta})\geqslant \inf \frac{x^*}{\|x^*\|}(B)= \frac{x^*}{\|x^*\|}(\tilde{x})-\tilde{r}>0. 
 \end{equation}
  Hence, by Lemma $\ref{phelp's lemma},$ 
\begin{equation}\label{phel}
\Big\|x^*_0-\frac{x^*}{\|x^*\|}\Big\|<2\varepsilon + 2\delta.
\end{equation}  
Finally by using  $\eqref{phel}$, we have
\begin{equation}
\Big\|x^*_0 - x^*\Big\|\leqslant \Big\|x^*_0-\frac{x^*}{\|x^*\|}\Big\| + \Big\|\frac{x^*}{\|x^*\|} -x^*\Big\|<2\varepsilon + 2\delta +(1-\|x^*\|)<  3 \varepsilon.
\end{equation}   
  Hence, diam$(S)<6\varepsilon$.

(iii) $\Rightarrow$ (ii). Straightforward.

(ii) $\Rightarrow$ (iii). Let $\varepsilon>0.$ Then there exists $x^*_0\in S_{X^*}$ satisfying the condition in (ii). Let $C$ be a subset in $B_{X^{**}}$ such that $\inf x^*_0(C)> \varepsilon.$ 
 Choose $\delta >0$ such that $\varepsilon < \delta < \inf x^*_0(C)$. Consider $$A=\{ x\in B_X : x^*_0(x)>\delta\}.$$ 
Then $C\subset \overline{A}^{w^*}.$
 Indeed, let $x\in C\subset B_{X^{**}} = \overline{B_X}^{w^*}$. Then there exists a net $(x_{\lambda})$ in $B_X$ such that $w^*$-$\lim \limits_{\lambda} x_{\lambda}= x.$ Since $x^*_0(x)>\delta,$ then there exists $\lambda_0$ such that $x^*_0(x_{\lambda})>\delta$ for all $\lambda \geqslant \lambda_0.$ That gives
  $  x_{\lambda} \in A $ for all $\lambda \geqslant \lambda_0.$ Hence $x\in \overline{A}^{w^*}.$
  
  Also, $A$ is a subset of $B_X$ with $\inf x^*_0(A)>\varepsilon$.
 Then by (ii), there exists closed ball $B= \tilde{x} + \tilde{r} B_X$ in $X$ such that $A\subset B$ and $\inf x^*_0(B)>0$.
 Therefore  
 $$C\subset \overline{A}^{w^*} \subset \tilde{x} + \tilde{r} B_{X^{**}} \ \ \mathrm{and} \ \  \inf x^*_0 (\tilde{x} + \tilde{r} B_{X^{**}})>0.$$
 

\end{proof}

Taking into account Proposition $\ref{A1}$ and Theorem $\ref{ch w star bdp}$, we get the following.

\begin{corollary}\label{ch bdp}
For a Banach space $X,$ the following are equivalent :
\begin{enumerate} 
\item $X$ has $BDP.$
\item Given $\varepsilon>0,$ there exists $x_0\in S_{X}$ such that 
for every subset $C$ of $B_{X^*}$ with $\inf x_0(C)> \varepsilon,$ there exists a closed ball $B$ in $X^*$ satisfying $C\subset B$ and $\inf x_0(B)>0.$
\end{enumerate}
\end{corollary}

\begin{proof}
(ii) $\Rightarrow$ (i) follows directly from Proposition $\ref{A1}$ and Theorem $\ref{ch w star bdp}$.

(i) $\Rightarrow$ (ii). Let $\varepsilon>0.$ Since $X^{**}$ has $w^*$-$BDP$ (by Proposition $\ref{A1}$), there exists a $w^*$-slice $S(B_{X^{**}},x_0^*,2\alpha)$ of $B_{X^{**}}$ with diameter less than $\frac{\varepsilon}{2}.$ By the Bishop-Phelps theorem \cite[Theorem 3.1]{DGZ}, we may assume without loss of generality that $x_0^*$ is a norm attaining functional. Choose $x_0 \in S_X$ such that $x_0^*(x_0)=1$. Then following a similar argument as in the proof of (i) $\Rightarrow$ (ii) of Theorem $\ref{ch w star bdp}$, we conclude that $x_0$ satisfies the condition in (ii).
\end{proof}


\begin{lemma}
\cite[Lemma 4.2]{CL1}
In a Banach space $X$ for every $x^*_0\in S_{X^*}$ and for any $w^*$-neighborhood $U$ of $x^*_0$ in $B_{X^*},$ there exists a $w^*$-neighborhood $V$ of $x^*_0$ in $B_{X^*}$ such that $V=\{x^*\in B_{X^*} : x^*(x_i)>\gamma_i>0, \ i=1,2,\ldots,n\}$, for some $x_1,x_2,\ldots,x_n$ in $X$ and $V\subset U.$
\end{lemma}

\begin{theorem}\label{ch w star bhp}
For a Banach space $X,$ the following are equivalent :
\begin{enumerate}
\item $X^*$ has $w^*$-$BHP$.
\item Given $\varepsilon>0,$ there exists $x^*_0\in S_{X^*}$ such that 
for every subset $C$ of $B_X$ with $\inf x^*_0(C)> \varepsilon,$  there exist closed balls $B_1,B_2,\ldots,B_n$ in $X$ satisfying $C\subset \overline{\mathrm{co}} (\bigcup\limits_{i=1}^{n} B_i)$ and $\inf x^*_0(\overline{\mathrm{co}} (\bigcup\limits_{i=1}^{n} B_i))> 0$.
\item Given $\varepsilon>0,$ there exists $x^*_0\in S_{X^*}$ such that
for every subset $C$ of $B_{X^{**}}$ with $\inf x^*_0(C)> \varepsilon,$ there exist closed balls $B_1,B_2,\ldots,B_n$ in $X^{**}$ with center in $X$ satisfying $C\subset \overline{\mathrm{co}} (\bigcup\limits_{i=1}^{n} B_i)$ and $\inf x^*_0(\overline{\mathrm{co}} (\bigcup\limits_{i=1}^{n} B_i))> 0$.

 
\end{enumerate}
\end{theorem}

\begin{proof}
(i) $\Rightarrow$ (ii). Let $0<\varepsilon <1.$ Since $X^*$ has $w^*$-$BHP$, then there exists a nonempty $w^*$-open subset of $B_{X^*}$
$$V=\{x^*\in B_{X^*} : x^*(x_i)>\gamma_i>0, x_i\in X, \ i=1,2,\ldots,n\}$$ 
for some $x_1 , \ldots, x_n \in X$, with diameter less than $\frac{\varepsilon}{3}.$ We choose $x^*_0\in V\bigcap S_{X^*}$. Then $V \subset B(x^*_0,\frac{\varepsilon}{3}).$
We aim to show that $x^*_0$ satisfies the condition in (ii).

Consider a subset $C\subset B_X$ with
$\inf x^*_0(C)> \varepsilon.$ 
Since $x^*_0\in V,$ then for each $i=1,\ldots,n$, we have $x^*_0(x_i)>\gamma_i$ and therefore we can choose $\alpha_i>0$ and $m>0$ such that $x^*_0(x_i)>\alpha_i> \alpha_i-\frac{1}{m}> \gamma_i>0$ for all $i = 1, \ldots, n$.
Also since $x^*_0\neq 0$, then we can choose $x_0 \in X$ such that $\|x_0\|<\frac{\varepsilon}{3}$ and $x^*_0(x_0)>0.$ Also choose $0<\eta < x^*_0(x_0)$.
 Consider $$K=\overline{\mathrm{co}} (B(mx_1, m\alpha_1) \bigcup \ldots \bigcup B(mx_n, m\alpha_n)\bigcup B(x_0,\eta)).$$ 
 Observe that, 
 $$\inf x^*_0(B(mx_i, m\alpha_i))= x^*_0(mx_i)- m\alpha_i>0 \quad \forall i=1,\ldots,n$$
 $$\mathrm{and} \ \inf x^*_0(B(x_0,\eta))= x^*_0(x_0)- \eta>0.$$
 Therefore,  $\inf x^*_0(K)\geqslant \min\{x^*_0(x_0)- \eta, x^*_0(mx_1)- m\alpha_1, \ldots, x^*_0(mx_n)- m\alpha_n\}  >0.$ \\
We claim that  $C\subset K.$ \\
If not, there exists $y_0\in C\setminus K.$ By Hahn-Banach separation theorem there exists $x_1^*\in S_{X^*}$ such that $\inf x_1^*(K) > x_1^*(y_0)\geqslant -1.$ Therefore, for each $i=1,\ldots,n$, $$x_1^*(mx_i)-m\alpha_i = \inf x_1^* (B(mx_i,m\alpha_i))\geqslant \inf x_1^*(K)>-1$$
That gives, $x_1^*(x_i)>\alpha_i-\frac{1}{m}>\gamma_i$ for all $i=1,\ldots,n$ and hence $x_1^*\in V.$ Thus, $\|x_1^*-x^*_0\|<\frac{\varepsilon}{3}$ and hence we have
\begin{equation}\label{chap 6 equ}
\inf x_1^*(K)> x_1^*(y_0)=x^*_0(y_0)-(x^*_0-x_1^*)(y_0)\geqslant \inf x^*_0(C)-\frac{\varepsilon}{3}>\varepsilon- \frac{\varepsilon}{3} = \frac{2\varepsilon}{3}.
\end{equation}
 Again,
\begin{equation}\label{chap 6 equation}
\inf x_1^*(K) \leqslant x_1^*(x_0) \leqslant \|x_0\| < \frac{\varepsilon}{3}<\frac{2\varepsilon}{3}.
\end{equation} 
   Therefore from $\eqref{chap 6 equ}$ and $\eqref{chap 6 equation}$ we get a contradiction. Hence, our claim is true.

(ii) $\Rightarrow$ (i). Let $\varepsilon>0.$ Then there exists $x^*_0\in S_{X^*}$ satisfying the condition in (ii). Choose $0<\delta <\frac{\varepsilon}{2}$.
 Define $$K_{\delta}=\{x\in B_X: x^*_0(x)>\varepsilon + \delta\}.$$ Then $\inf x^*_0(K_{\delta})>\varepsilon.$ Therefore by (ii), we can choose closed balls $B_i= B(\tilde{x_i},\tilde{r_i})$ $(1\leqslant i \leqslant n)$
in $X$ such that $K_{\delta}\subset \overline{\mathrm{co}} (\bigcup\limits_{i=1}^{n} B_i)$ and $\inf x^*_0(\overline{\mathrm{co}} (\bigcup\limits_{i=1}^{n} B_i))> 0$. Then 
\begin{equation}\label{ju}
x^*_0(\tilde{x_i})-\tilde{r_i}=\inf x^*_0(B_i)> 0 \quad \forall i=1,\ldots,n.
\end{equation}
Choose $z_0\in S_X$ such that $x^*_0(z_0)>1-\varepsilon.$ 
 Consider the $w^*$-open set 
  $$V=\bigcap\limits_{i=1}^{n} S(B_{X^*}, \frac{\tilde{x_i}}{\|\tilde{x_i}\|}, 1-\frac{\tilde{r_i}}{\|\tilde{x_i}\|}) \bigcap S(B_{X^*}, z_0,\varepsilon).$$ 
 Observe that $x^*_0\in V$ (by $\eqref{ju}$) and hence $V$ is nonempty.\\
 Let $x^*\in V.$ Then $x^*(\tilde{x_i})>\tilde{r_i}$ for all $i=1,\ldots,n$. Therefore, 
 \begin{equation}\notag
 \begin{split}
 \inf x^*(K_{\delta}) \geqslant \inf x^* (\overline{\mathrm{co}} (\bigcup\limits_{i=1}^{n} B_i))
 \geqslant \min\{\inf x^*(B_1), \ldots, \inf x^*(B_n) \}\hspace{1.5 cm}\\
  \geqslant \min\{x^*(\tilde{x_1}) - \tilde{r_1}, \ldots, x^*(\tilde{x_n}) - \tilde{r_n}\}\hspace{1 cm}\\
  >0.\hspace{6.8 cm}
 \end{split}
 \end{equation}
Hence, by  Lemma $\ref{phelp's lemma},$ we have 
\begin{equation}\label{bhp ball eq1}
\left\|x^*_0-\frac{x^*}{\|x^*\|}\right\|<2\varepsilon + 2 \delta.
\end{equation} 
Also since $x^*\in V$, then
\begin{equation}\label{bhp ball eq2}
\|x^*\|\geqslant x^*(z_0)>1-\varepsilon.
\end{equation}
Finally taking into account $\eqref{bhp ball eq1}$ and $\eqref{bhp ball eq2}$ we have
$$\|x^*_0-x^*\|\leqslant \left\|x^*_0-\frac{x^*}{\|x^*\|}\right\|+\left\|\frac{x^*}{\|x^*\|}-x^*\right\|<2\varepsilon + 2\delta + (1-\|x^*\|)< 4\varepsilon.$$
Therefore, diam$(V)<8 \varepsilon.$

(iii) $\Rightarrow$ (ii). Straightforward.

(ii) $\Rightarrow$ (iii). Let $\varepsilon>0.$ Then there exists $x^*_0\in S_{X^*}$ satisfying the condition in (ii). Let $C$ be a subset in $B_{X^{**}}$ such that $\inf x^*_0(C)> \varepsilon.$ 
Choose $\delta >0$ such that $\varepsilon < \delta < \inf x^*_0(C)$. Consider $$A=\{ x\in B_X : x^*_0(x)>\delta\}.$$ 
Then $C\subset \overline{A}^{w^*}.$
 Indeed, let $x\in C\subset B_{X^{**}} = \overline{B_X}^{w^*}$. Then there exists a net $(x_{\lambda})$ in $B_X$ such that $w^*$-$\lim \limits_{\lambda} x_{\lambda}= x$. Since $x^*_0(x)>\delta,$ then there exists $\lambda_0$ such that $x^*_0(x_{\lambda})>\delta$ for all $\lambda \geqslant \lambda_0.$ That gives
  $  x_{\lambda} \in A $ for all $\lambda \geqslant \lambda_0.$ Hence $x\in \overline{A}^{w^*}.$
  
  Also, $A$ is a subset of $B_X$ with $\inf x^*_0(A)>\varepsilon$.
 Then by (ii), there exist closed balls $B_i= \tilde{x_i} + \tilde{r_i} B_X$ $(i=1,2,\ldots,n)$  in $X$ such that $A\subset \overline{\mathrm{co}} (\bigcup\limits_{i=1}^{n} B_i)$ and $\inf x^*_0(\overline{\mathrm{co}} (\bigcup\limits_{i=1}^{n} B_i))> 0$.
 Therefore, 
 \begin{center}
$C\subset \overline{A}^{w^*} \subset \overline{\mathrm{co}} \bigcup\limits_{i=1}^{n} (\tilde{x_i} + \tilde{r_i} B_{X^{**}})$ and $\inf x_0^*(\overline{\mathrm{co}} \bigcup\limits_{i=1}^{n} (\tilde{x_i} + \tilde{r_i} B_{X^{**}}))>0$.
 \end{center}

\end{proof}

Taking into account Proposition $\ref{A1}$ and Theorem $\ref{ch w star bhp}$, we get the following.

\begin{corollary}
For a Banach space $X,$ the following are equivalent :
\begin{enumerate}
\item $X$ has $BHP.$
\item Given $\varepsilon>0,$ there exists $x_0\in S_{X}$ such that 
for every subset $C$ of $B_{X^*}$ with $\inf x_0(C)> \varepsilon,$  there exist balls $B_1,B_2,\ldots,B_n$ in $X^*$ satisfying $C\subset \overline{\mathrm{co}} (\bigcup\limits_{i=1}^{n} B_i)$ and $\inf x_0(\overline{\mathrm{co}} (\bigcup\limits_{i=1}^{n} B_i))> 0$.

\end{enumerate}
\end{corollary}

\begin{proof}
Let $0<\varepsilon <1.$ Since $X^{**}$ has $w^*$-$BHP$ (by Proposition $\ref{A1}$), there exists a nonempty $w^*$-open subset $V$ of $B_{X^{**}}$
 with diameter less than $\frac{\varepsilon}{3}.$ By Goldstine’s theorem, we can choose $x_0 \in  V\bigcap S_{X}$ such that $V \subset B(x_0,\frac{\varepsilon}{3}).$ Then following a similar argument as in the proof of (i) $\Rightarrow$ (ii) of Theorem $\ref{ch w star bhp}$, we conclude that $x_0$ satisfies the condition in (ii).
\end{proof}


\begin{proposition}\label{CL result 1}
\cite[Theorem 2.1]{CL1}
Let $X$ be a Banach space and $C$ be a bounded subset of $X^{**}.$ Then there exists a ball $B$ in $X^{**}$ with center in $X$ such that $C\subset B$ and $0\notin B$ if and only if $\overline{\mathrm{co}}^{w^*} B_1(C)\neq B_{X^*}$, where $B_1(C)= \{x^*\in B_{X^*}: x^{**}(x^*)\leqslant 0$ for some $x^{**}\in C\}.$
\end{proposition}

By using the same technique as in the proof of Proposition $\ref{CL result 1}$, we establish the following conclusion.
 
\begin{proposition}\label{CLcoro result 1}
Let $X$ be a Banach space and $C$ be a bounded subset of $X.$ Then there exists a ball $B$ in $X$ such that $C\subset B$ and $0\notin B$ if and only if $\overline{\mathrm{co}}^{w^*} B_1(C)\neq B_{X^*}$, where $B_1(C)= \{x^*\in B_{X^*}: x^*(x)\leqslant 0$ for some $x\in C\}.$
\end{proposition}

\begin{theorem}\label{ch w star bscsp}
Let $X$ be a Banach space. Consider the following statements.
\begin{enumerate}
\item $X^*$ has $w^*$-$BSCSP.$

\item Given $\varepsilon>0,$ there exists $x^*_0 \in B_{X^*}\setminus \{0\}$ such that
for every subset $C$ of $B_X$ with $\inf x^*_0(C)>\varepsilon,$  there exist balls $B_1,B_2,\ldots,B_n$ in $X$ satisfying $C\subset \bigcup\limits_{i=1}^{n} B_i$ and $0\notin \bigcup\limits_{i=1}^{n} B_i.$
\item Given $\varepsilon>0,$ there exists $x^*_0\in B_{X^*}\setminus \{0\}$ such that
for every subset $C$ of $B_{X^{**}}$ with $\inf x^*_0(C)>\varepsilon,$  there exist balls $B_1,B_2,\ldots,B_n$ in $X^{**}$ with center in $X$ satisfying $C\subset \bigcup\limits_{i=1}^{n} B_i$ and $0\notin \bigcup\limits_{i=1}^{n} B_i.$

\end{enumerate}
Then (i) $\Rightarrow$ (ii) $\Leftrightarrow$ (iii).
 \end{theorem}

\begin{proof}


(i) $\Rightarrow$ (ii). Let $0<\varepsilon <1.$ Since $X^*$ has $w^*$-$BSCSP,$ then there exists a convex combination of $w^*$-slices
$S=\sum\limits_{i=1}^{n} \lambda_i S_i$ of $B_{X^*}$  with diameter less than $\frac{\varepsilon}{2}.$ We choose $x^*_0  \in S\setminus \{0\}$ such that $S \subset B(x^*_0,\frac{\varepsilon}{2}).$ We aim to show that $x^*_0$ satisfies the condition in (ii).

Consider a subset $C$ of $B_X$ with $\inf x^*_0(C)>\varepsilon.$ 
 Observe that for any $c\in C$ and $x^*\in S$
$$x^*(c)
= x^*_0(c)-( x^*_0 - x^* )(c)
> \varepsilon - \|x^*_0-x^*\| 
\geqslant \varepsilon - \frac{\varepsilon}{2}
= \frac{\varepsilon}{2} .$$
Therefore,
\begin{equation}\label{ball sep bscspeq2}
\inf x^* (C) \geqslant \frac{\varepsilon}{2} >0 \quad \forall x^*\in S.
\end{equation}
We claim that  $C=\bigcup\limits_{i=1}^{n} C_i$ where $C_i= \{ x\in C : x^*(x)> 0 \ \forall x^*\in S_i\}.$\\
If not, there exists $x\in C \setminus \bigcup\limits_{i=1}^{n}C_i.$ Then for each $i=1,\ldots,n$ there exist $x_i^* \in S_i$ such that 
\begin{equation}\label{ball sep bscspeq}
x_i^*(x)\leqslant 0 \quad \forall i=1,\ldots,n.
\end{equation}
 Let $x^*= \sum\limits_{i=1}^{n} \lambda_i x_i^*.$ Then $x^*\in S$ and by $\eqref{ball sep bscspeq}$, we have $x^*(x)\leqslant 0$. This contradicts $\eqref{ball sep bscspeq2}$. Hence, our claim is true.

  For each $i=1,\ldots,n$, we define
  \begin{center}
 $B_1(C_i)=\{x^*\in B_{X^*} : x^*(x)\leqslant 0 $ for some $x\in C_i \}.$  
  \end{center}
  Then $B_1(C_i) \bigcap S_i =\emptyset$ for all $i=1,\ldots,n$.  
  By Proposition $\ref{CLcoro result 1}$, for each $i=1,\ldots,n$, there exist balls $B_i$  in $X$ such that 
$C_i \subset B_i,$   $  0\notin B_i.$  
Thus, $C \subset \bigcup\limits_{i=1}^{n} B_i$ and $ 0\notin \bigcup\limits_{i=1}^{n} B_i.$ 

(iii) $\Rightarrow$ (ii). Straightforward.

(ii) $\Rightarrow$ (iii). Let $\varepsilon>0.$ Then there exists $x^*_0 \in B_{X^*}\setminus \{0\}$ satisfying the condition in (ii).
  Let $C$ be a subset of $B_{X^{**}}$ such that $\inf x^*_0(C)> \varepsilon.$ 
   Choose $\delta >0$ such that $\varepsilon < \delta < \inf x^*_0(C)$.
   Consider $$A=\{ x\in B_X : x^*_0(x)>\delta\}.$$
Then  $C\subset \overline{A}^{w^*}.$ 
Indeed, let $x\in C\subset B_{X^{**}} = \overline{B_X}^{w^*}$. Then there exists a net $(x_{\lambda})$ in $B_X$ such that $w^*$-$\lim\limits \ x_{\lambda}= x.$ Since $x^*_0(x)>\delta,$ then there exists $\lambda_0$ such that $x^*_0(x_{\lambda})>\delta$ for all $\lambda \geqslant \lambda_0.$ That gives $x_{\lambda} \in A$  for all $\lambda \geqslant \lambda_0.$ Hence $x\in \overline{A}^{w^*}.$

Also, $A$ is a subset of $B_X$ with $\inf x^*_0(A)>\varepsilon.$ Then by (ii), there exist balls $B_i= x_i + r_i B_X$ $(i=1,2,\ldots,n)$ in $X$ such that $A\subset \bigcup\limits_{i=1}^{n} B_i$ and $0\notin \bigcup\limits_{i=1}^{n} B_i$. Therefore,

\hspace{2.5 cm} $C\subset \overline{A}^{w^*} \subset  \bigcup\limits_{i=1}^{n} (x_i + r_i B_{X^{**}})$ and $0\notin  \bigcup\limits_{i=1}^{n} (x_i + r_i B_{X^{**}}).$
\end{proof}

\begin{remark}
The validity of the implication (ii) $\Rightarrow$ (i) of Theorem $\ref{ch w star bscsp}$ is not known in general.
\end{remark}

Taking into account Proposition $\ref{A1}$ and Theorem $\ref{ch w star bscsp}$, we get the following.

\begin{corollary}
Suppose $X$ is a Banach space with $BSCSP.$ Then 
for $\varepsilon>0,$ there exists $x_0 \in B_{X}\setminus \{0\}$ such that
for every bounded set $C$ of $X^*$ with $\inf x_0(C)>\varepsilon,$ there exist balls $B_1,B_2,\ldots,B_n$ in $X^*$ satisfying $C\subset \bigcup\limits_{i=1}^{n} B_i$ and $0\notin \bigcup\limits_{i=1}^{n} B_i.$
\end{corollary}

\begin{proof}
Let $0<\varepsilon <1.$ Since $X^{**}$ has $w^*$-$BSCSP$ (by Proposition $\ref{A1}$), there exists a convex combination of $w^*$-slices
$S=\sum\limits_{i=1}^{n} \lambda_i S_i$ of $B_{X^{**}}$  with diameter less than $\frac{\varepsilon}{2}.$ By Goldstine’s theorem, we can choose $x_0  \in B_X \bigcap (S\setminus \{0\})$ such that $S \subset B(x_0,\frac{\varepsilon}{2}).$ Then following a similar argument as in the proof of (i) $\Rightarrow$ (ii) of Theorem $\ref{ch w star bscsp}$, we conclude that $x_0$ satisfies the required condition.
\end{proof}



\section{Semi Denting, semi PC and semi SCS Points}
\begin{definition}
Let $X$ be a Banach space. A point $x\in B_{X}$ is said to be  
\begin{enumerate}
	\item\cite{CL2} \emph{semi denting point}  of $B_{X}$ if for every $\e > 0$, there exists slice $S$ of $B_{X}$ such that  $S\subset B(x,\varepsilon)$. 
\item \emph{semi Point of Continuity} (semi PC) of $B_{X}$ if for every $\varepsilon >0,$ there exists a weakly open set $V$ in $B_{X}$ such that $V\subset B(x,\varepsilon).$
\item \emph{semi Small Combination of Slice} (semi SCS) point of $B_{X}$ if for every $\varepsilon >0,$ there exists a convex combination of slices $\sum\limits_{i=1}^{n} \lambda_i S_i$ in $B_{X}$ such that $\sum\limits_{i=1}^{n} \lambda_i S_i\subset B(x,\varepsilon).$
\end{enumerate}
Analogously, we can define semi $w^*$-denting point \index{Semi $w^*$-denting point}, semi $w^*$-PC \index{Semi $w^*$-PC} and semi $w^*$-SCS \index{Semi $w^*$-SCS} of $B_{X^*}$.
\end{definition}

\begin{remark}
It is clear that the existence of a semi denting (respectively, semi PC) point in $B_X$ implies that $X$ has the BDP (respectively, BHP).
However, it is not known whether the converse holds in general. It is true that $X$ has $BSCSP$ if and only if 0 is a SCS point in $B_X.$ Indeed, if $X$ has $BSCSP$, then for $\varepsilon >0$ there exists a convex combination of slices $T$ of $B_X$ with diameter less than $\varepsilon$. The convex combination of slices $T'=\frac{1}{2} T + \frac{1}{2} (-T)$ of $B_X$ then contains 0 and has diameter less than $\varepsilon$.  It remains unknown whether $BSCSP$ ensures the existence of any nonzero SCS point or nonzero semi SCS point in $B_X$.
\end{remark}

\begin{remark}\label{dent-ext}
If a point $x_0^* \in S_{X^*}$ is a semi $w^*$-denting point of $B_{X^*}$, then there exists a sequence $(x_n^*)_n$ of extreme points in $S_{X^*}$ that converges to $x_0^*$ in norm.
Indeed,
  let $x_0^*\in S_{X^*}$ be a semi $w^*$-denting point of $B_{X^*}$. Then there exists a sequence $(S_n)_n$ of $w^*$-slices of $B_{X^*}$ such that $S_n \subset B(x_0^*, \frac{1}{n})$ for all $n\in \mathbb{N}$. By the Krein-Milman Property (KMP) and the extremal structure of slices (that is, if $S$ is a slice of $B_{X^*}$ and $\sum\limits_{i=1}^{l} \lambda_i a_i \in S$ with $0<\lambda_i < 1$ and $\sum\limits_{i=1}^{l} \lambda_i=1$, then $a_i \in S$ for some $i$), each $S_n$ contains an extreme point $x_n^*$ (say) of $B_{X^*}$. Consequently, $(x_n^*)_n$ converges to $x_0^*$ in norm. However, unlike the case of $w^*$-denting points, it is not known whether every semi $w^*$-denting point must be an extreme point.
\end{remark}

\begin{theorem}\label{w star semi denting G}
\cite[Theorem 4.1]{G}
Let $X$ be a Banach space and $x_0^*\in S_{X^*}$. Then $x_0^*$ is semi $w^*$-denting point of $B_{X^*}$ if and only if for every bounded set $C\subset X$ with $\inf x^*_0(C)>0$, there exists a  closed ball $B$ in $X$ such that $C\subset B$ and $0\notin B.$
\end{theorem}

\begin{theorem}\label{w star semi denting char}
Let $X$ be a Banach space, $x^*_0\in S_{X^*}$ and $H=\{x\in X : x^*_0(x)=0\}$. Then the following are equivalent :
\begin{enumerate}
\item $x^*_0$ is semi $w^*$-denting point of $B_{X^*}.$
\item For every bounded set $C$ of $X^{**}$ with $\inf x^*_0(C)> 0,$  there exists closed ball $B$ in $X^{**}$ with center in $X$ such that $C\subset B$ and $0\notin B.$

\item For every bounded set $C$ of $X$ with $\inf x^*_0(C)> 0,$
there exists a family of closed balls $\{B_i:i\in I\}$ in $X$ such that $C\subset \bigcap\limits_{i\in I} B_i $ 
and $(\bigcap \limits_{i\in I} B_i) \bigcap H =\emptyset.$
\end{enumerate}
\end{theorem}

\begin{proof}
(ii) $\Rightarrow$ (i). Follows from Theorem $\ref{w star semi denting G}$.

(i) $\Rightarrow$ (ii). 
Let $C$ be a bounded set in $X^{**}$ such that $\inf x^*_0(C) >0.$ Thus, $C\subset rB_{X^{**}}$ for some $r>0.$ Choose $\varepsilon >0$ such that $0<\varepsilon  < \inf x^*_0(C)$. Consider $$A=\{ x\in rB_X : x^*_0(x)>\varepsilon\}.$$ 
Then $C\subset \overline{A}^{w^*}.$
 Indeed, let $x\in C\subset rB_{X^{**}} = \overline{rB_X}^{w^*}$. Then there exists a net $(x_{\lambda})$ in $rB_X$ such that $w^*$-$\lim \limits_{\lambda} x_{\lambda}= x.$ Since $x^*_0(x)>\varepsilon,$ then there exists $\lambda_0$ such that $x^*_0(x_{\lambda})>\varepsilon$ for all $\lambda \geqslant \lambda_0.$ That gives
  $  x_{\lambda} \in A $ for all $\lambda \geqslant \lambda_0.$ Hence $x\in \overline{A}^{w^*}.$
  
  Also, $A$ is a bounded set of $X$ with $\inf x^*_0(A) \geqslant \varepsilon >0$.
 Then by Theorem $\ref{w star semi denting G}$, there exists a closed ball $B= \tilde{x} + \tilde{r} B_X$ in $X$ such that $A\subset B$ and $0\notin B$. Therefore  
 $$C\subset \overline{A}^{w^*} \subset \tilde{x} + \tilde{r} B_{X^{**}} \ \ \mathrm{and} \ \ 0\notin \tilde{x} + \tilde{r} B_{X^{**}}.$$

(i) $\Rightarrow$ (iii).
Let $C$ be a bounded set in $X$ with $\inf x^*_0(C)>0$.
Let $x\in H$. Then $C-x$ is a bounded set in $X$ with $\inf x^*_0(C-x)>0 $. By Theorem $\ref{w star semi denting G}$, there exists a closed ball $B=B(z_x, r_x)$ in $X$ such that 
$C-x\subset B(z_x, r_x)$ and $0\notin B(z_x, r_x)$.
Therefore, 
$C\subset B(x+z_x, r_x)$ and $x\notin B(x+z_x, r_x)$.
Let $B_x= B(x+z_x, r_x)$.
Hence, 
$C\subset \bigcap\limits_{x\in H} B_x $ 
and $(\bigcap \limits_{x\in H} B_x) \bigcap H =\emptyset.$

(iii) $\Rightarrow$ (i). 
Let $C$ be a bounded set in $X$ with $\inf x^*_0(C) >0$. 
 Therefore by (iii), we can choose a family $\{B_i:i\in I\}$ of closed balls in $X$ such that $C\subset \bigcap\limits_{i\in I} B_i $ 
and $(\bigcap \limits_{i\in I} B_i) \bigcap H =\emptyset.$ Since $0\in H$, then there exists $j\in I$ such that $0\notin B_j$. Hence, $C\subset B_j$ and $0\notin B_j$.
\end{proof}

\begin{corollary}\label{w star semi denting coro}
Let $X$ be a Banach space, $x_0\in S_X$ and $H=\{x^*\in X^* : x^*(x_0)=0\}$. Then
the following are equivalent :
\begin{enumerate}
\item $x_0$ is semi denting point of $B_X$.
\item For every bounded set $C$ of $X^*$ with $\inf x_0(C)> 0,$  there exists closed ball $B$ in $X^*$ such that $C\subset B$ and $0\notin B.$

\item For every bounded set $C$ of $X^*$ with $\inf x_0(C)> 0,$
there exists  a family of closed balls $\{B_i:i\in I\}$ in $X$ such that 
$C\subset \bigcap\limits_{i\in I} B_i $ 
and $(\bigcap \limits_{i\in I} B_i) \bigcap H =\emptyset.$
\end{enumerate}
\end{corollary}

\begin{proof}
By the arguments used in the proof of Proposition $\ref{A1}$, it follows that
 $x_0\in S_X$ is semi denting point of $B_X$ if and only if $x_0$ is semi $w^*$-denting point of  $B_{X^{**}}$. 
 The rest follows from Theorem $\ref{w star semi denting char}$.
\end{proof}

\begin{remark}
The fact that $x_0\in S_X$ is semi denting point of $B_X$ if and only if $x_0$ is semi $w^*$-denting point of  $B_{X^{**}}$ was already known and proved in \cite[Lemma 5.1]{G}.
\end{remark}

\begin{remark}
The equivalence (i) $\Leftrightarrow$ (ii) in Corollary $\ref{w star semi denting coro}$ can also be derived from Theorem $\ref{w star semi denting G}$, as shown in \cite[Theorem 5.2]{G}.
\end{remark}

\begin{theorem}\label{w star semi pc char}
Let $X$ be a Banach space, $x^*_0\in S_{X^*}$ and $H=\{x\in X : x^*_0(x)=0\}$. Then
 the following are equivalent :
\begin{enumerate}
\item $x^*_0$ is semi $w^*$-PC of $B_{X^*}.$
\item For every bounded set $C$ of $X$ with $\inf x^*_0(C)>0,$ there exist  closed balls $B_1,\ldots,B_n$ in $X$ such that $C\subset \overline{\mathrm{co}} (\bigcup\limits_{i=1}^{n} B_i)$ and $0\notin \overline{\mathrm{co}} (\bigcup\limits_{i=1}^{n} B_i).$
\item For every bounded set $C$ of $X^{**}$ with $\inf x^*_0(C)> 0,$  there exist closed balls $B_1,B_2,\ldots,B_n$ in $X^{**}$ with center in $X$ such that $C\subset \overline{\mathrm{co}} (\bigcup\limits_{i=1}^{n} B_i)$ and $0\notin \overline{\mathrm{co}} ( \bigcup\limits_{i=1}^{n} B_i).$

\item For every bounded set $C$ of $X$ with $\inf x^*_0(C)>0,$
there exists a family $\{K_i:i\in I\}$, where each $K_i$ is closed convex hull of finitely many closed balls 
 in $X$, such that $C\subset \bigcap\limits_{i\in I} K_i $ 
and $(\bigcap \limits_{i\in I} K_i) \bigcap H =\emptyset.$
\end{enumerate}
\end{theorem}

\begin{proof}
The equivalence (ii) $\Leftrightarrow$ (iv) follows from arguments analogous to those in Theorem $\ref{w star semi denting char}$ (i) $\Leftrightarrow$ (iii).

(i) $\Rightarrow$ (ii).
Consider bounded set $C$ of $X$ with
$\inf x^*_0(C) >0.$  Without loss of generality we can assume that $C\subset B_{X}.$
Otherwise, we can choose $m>0$ such that $\frac{1}{m} C\subset B_X$. Also $\inf x^*_0(C) >0$ implies $\inf x^*_0(\frac{1}{m} C) >0.$ Now, if there exist closed balls $B(z_1,r_1), \ldots, B(z_n,r_n)$ in $X$ such that $\frac{1}{m} C\subset \overline{\mathrm{co}} (\bigcup\limits_{i=1}^{n} B(z_i,r_i))$ and $0\notin \overline{\mathrm{co}} (\bigcup\limits_{i=1}^{n} B(z_i,r_i))$, then $C\subset \overline{\mathrm{co}} (\bigcup\limits_{i=1}^{n} B(mz_i,mr_i))$ and $0\notin \overline{\mathrm{co}} (\bigcup\limits_{i=1}^{n} B(mz_i,mr_i))$.

Choose $0<\varepsilon<\inf x^*_0(C)$.
 Then $\varepsilon <1.$
 Since $x^*_0$ is semi $w^*$-PC of $B_{X^*},$ then there exists a nonempty $w^*$-open set $V$
  of $B_{X^*}$ such that $V\subset B(x^*_0,\frac{\varepsilon}{3}).$
 Then proceeding with similar technique as in the proof of (i)  $\Rightarrow$ (ii) of Theorem $\ref{ch w star bhp}$, we get closed balls $B_1,\ldots,B_n$ in $X$ such that $C\subset \overline{\mathrm{co}} (\bigcup\limits_{i=1}^{n} B_i)$ and $0\notin \overline{\mathrm{co}} (\bigcup\limits_{i=1}^{n} B_i).$
 
(ii) $\Rightarrow$ (i). Let $\varepsilon>0.$ Thus by (ii), we have, for every subset $C$ of $B_X$ with $\inf x^*_0(C)>\varepsilon$, there exist  closed balls $B_1,\ldots,B_n$ in $X$ such that $C\subset \overline{\mathrm{co}} (\bigcup\limits_{i=1}^{n} B_i)$ and $0\notin \overline{\mathrm{co}} (\bigcup\limits_{i=1}^{n} B_i).$
Then proceeding with similar technique as in the proof of (ii) $\Rightarrow $ (i) of Theorem $\ref{ch w star bhp}$, we get a nonempty $w^*$-open set $V$ of $B_{X^*}$ such that $V\subset B(x^*_0, 3 \varepsilon)$.
Hence, $x^*_0$ is semi $w^*$-PC.

(iii) $\Rightarrow$ (ii). Straightforward.

(ii) $\Rightarrow$ (iii).  Let $C$ be a bounded set in $X^{**}$ such that $\inf x^*_0(C)> 0.$ Thus, $C\subset rB_{X^{**}}$ for some $r>0.$ Choose $0 < \delta < \inf x^*_0(C)$. Consider $A=\{ x\in rB_X : x^*_0(x)>\delta\}.$
Then, proceeding with a similar technique as in proof (ii) $\Rightarrow$ (iii) of Theorem $\ref{ch w star bhp}$, we obtain closed balls $B_1, \ldots, B_n$ in $X^{**}$ with centers in X such that
\begin{center}
$C\subset \overline{A}^{w^*} \subset \overline{\mathrm{co}} ( \bigcup\limits_{i=1}^{n} B_i)$ and $0\notin \overline{\mathrm{co}} ( \bigcup\limits_{i=1}^{n} B_i).$ 
 \end{center}

\end{proof}

\begin{corollary}
For a Banach space $X,$ the following are equivalent :
\begin{enumerate}
\item $x_0\in S_X$ is semi PC of $B_X$.
\item For every bounded set $C$ of $X^*$ with $\inf x_0(C)> 0,$ then there exist closed balls $B_1,B_2,\ldots,B_n$ in $X^*$ such that $C\subset \overline{\mathrm{co}} (\bigcup\limits_{i=1}^{n} B_i)$ and $0\notin \overline{\mathrm{co}} ( \bigcup\limits_{i=1}^{n} B_i).$

\item For every bounded set $C$ of $X^*$ with $\inf x_0(C)> 0,$
  then there exists  a family $\{K_i:i\in I\}$, where each $K_i$ is closed convex hull of finitely many closed balls 
  in $X^*$, such that 
$C\subset \bigcap\limits_{i\in I} K_i $ 
and $(\bigcap \limits_{i\in I} K_i) \bigcap H =\emptyset$, where $H=\{x^*\in X^* : x^*(x_0)=0\}$.
\end{enumerate}
\end{corollary}

\begin{proof}
By the arguments used in the proof of Proposition $\ref{A1}$, it follows that
 $x_0\in S_X$ is semi PC of $B_X$ if and only if $x_0$ is semi $w^*$-PC of  $B_{X^{**}}$. 
 The rest follows from Theorem $\ref{w star semi pc char}$.
\end{proof}

However, instead of considering semi $w^*$-PC, if we consider $w^*$-PC of $B_{X^*}$, then we have a stronger ball separation condition that was studied by Chen and Lin \cite[Corollary 4.4]{CL1}.

\begin{theorem} \label{ball sep pc weak}
\cite[Corollary 4.4]{CL1}
Let $X$ be a Banach space, $x_0^* \in S_{X^*}$ and $H=\{ x\in X : x_0^*(x)=0 \} $. 
Then the following are equivalent :
\begin{enumerate}
\item $x_0^* $ is $w^*$-PC of $B_{X^*}$.

\item For every bounded set $C$ of $X$ with $\inf x_0^*(C)> 0,$
 there exist closed balls $B_1,B_2,\ldots,B_n$ in $X$ such that 
 $C\subset \overline{\mathrm{co}} (\bigcup\limits_{i=1}^{n} B_i)$ and $\overline{\mathrm{co}} ( \bigcup\limits_{i=1}^{n} B_i)$ $\bigcap H = \emptyset.$

\end{enumerate}
\end{theorem}
By the arguments used in the proof of Proposition $\ref{A1}$, it follows that
 $x_0\in S_X$ is PC of $B_X$ if and only if $x_0$ is  $w^*$-PC of  $B_{X^{**}}$. Using Theorem $\ref{ball sep pc weak}$, the following holds.

\begin{corollary} \label{ball sep pc}
Let $X$ be a Banach space, $x_0 \in S_{X}$ and $H=\{ x^*\in X^* : x_0(x^*)=0 \}$. Then the following are equivalent:
\begin{enumerate}
\item $x_0$ is PC of $B_{X}$.
\item For every bounded set $C$ of $X^*$ with $\inf x_0(C)> 0,$ then there exist closed balls $B_1,B_2,\ldots,B_n$ in $X^*$ such that $C\subset \overline{\mathrm{co}} (\bigcup\limits_{i=1}^{n} B_i)$ and $\overline{\mathrm{co}} ( \bigcup\limits_{i=1}^{n} B_i)$ $\bigcap H = \emptyset.$

\end{enumerate}
\end{corollary}


In our next result, we will improve the following characterization of Property (II) in $X$ using the concept of semi $w^*$-PC of $B_{X^*}.$
\begin{theorem}\label{propii}
\cite[Theorem 4.6]{CL1}
A Banach space $X$ has Property (II) if and only if the set of $w^*$-PC of $B_{X^*}$ is norm dense in $S_{X^*}$.
\end{theorem}

\begin{theorem}\label{ch propii}
A Banach space $X$ has Property (II) if and only if every $x^*\in S_{X^*}$ is semi $w^*$-PC of $B_{X^*}.$
\end{theorem}

\begin{proof}
Suppose $X$ has Property (II). Let $x^*\in S_{X^*}$ and $\varepsilon >0.$ Then by Theorem $\ref{propii}$, there exists a $w^*$-PC $x^*_0\in S_{X^*}$ such that $\|x^*_0-x^*\|<\frac{\varepsilon}{2}.$ Since $x^*_0$ is $w^*$-PC of $B_{X^*}$, then there exists an $w^*$-open set $V,$  containing $x^*_0,$ in $B_{X^*}$ such that diameter of $V$ is less than $\frac{\varepsilon}{2}$.
 Thus, $V\subset B(x^*,\varepsilon).$ Hence, $x^*$ is a semi $w^*$-PC of $B_{X^*}.$

Conversely, suppose that every $x^*\in S_{X^*}$ is a semi $w^*$-PC of $B_{X^*}.$ Consider any bounded closed convex subset $C$ of $X$ and $x_0\notin C.$ Thus, by Hahn-Banach separation theorem, there exist  $x^*_0\in S_{X^*}$ such that $\inf x^*_0(C)>x^*_0(x_0).$ That gives $\inf x^*_0(C-x_0)>0.$ Since $x^*_0$ is semi $w^*$-PC of $B_{X^*}$, then by Theorem $\ref{w star semi pc char}$ there exist  closed balls $B_1,\ldots,B_n$ in $X$ such that $$(C-x_0)\subset \overline{\mathrm{co}} (\bigcup_{i=1}^{n} B_i) \ \ \mathrm{and} \ \ 0\notin \overline{\mathrm{co}} (\bigcup_{i=1}^{n} B_i).$$ Thus $C\subset \overline{\mathrm{co}} (\bigcup\limits_{i=1}^{n} (x_0 + B_i))$ and $x_0 \notin \overline{\mathrm{co}} (\bigcup\limits_{i=1}^{n} (x_0 + B_i)).$ Hence, $X$ has the Property (II).
\end{proof}

The $w^*$-analogues of MIP and Property (II), introduced in \cite{GGS}, \cite{CL1}, respectively, have been linked to denting points and Points of Continuity of $B_X$ (see \cite{G}, \cite{CL1}). We now recall only the $w^*$-version of Property (II) and its characterization in terms of PC.
A dual space $X^*$ is said to have \emph{$w^*$-(II) Property} if every bounded $w^*$-closed convex set in $X^*$ can be represented as an intersection of closed convex hulls of finitely many closed balls in $X^*$. It was shown in \cite{CL1} that  $X^*$ has $w^*$-(II) Property if and only if the set of PC of $B_X$ is norm dense in $S_{X}$.
Following techniques similar to the ones given in Theorem $\ref{ch propii}$, we get the following.

\begin{corollary}\label{wstarii}
A dual Banach space  $X^*$ has $w^*$-(II) Property if and only if every $x\in S_X$ is semi PC of $B_X$.
\end{corollary}


\begin{theorem}\label{wstarsemiscs}
Let $X$ be a Banach space, $x^*_0 \in B_{X^*}\setminus \{0\}$ 
 and $H=\{x\in X : x^*_0(x)=0\}$. 
 Consider the following statements.
\begin{enumerate}
\item $x^*_0$ is a semi $w^*$-SCS point of $B_{X^*}$.

\item For every bounded set $C$ of $X$ with $\inf x^*_0(C)>0,$  there exist closed balls $B_1,B_2,\ldots,B_n$ in $X$ such that $C\subset \bigcup\limits_{i=1}^{n} B_i$ and $0\notin \bigcup\limits_{i=1}^{n} B_i.$
\item For every bounded set $C$ of $X^{**}$ with $\inf x^*_0(C)>0,$  there exist closed balls $B_1,B_2,\ldots,B_n$ in $X^{**}$ with center in $X$ such that $C\subset \bigcup\limits_{i=1}^{n} B_i$ and $0\notin \bigcup\limits_{i=1}^{n} B_i.$
\item For every bounded set $C$ of $X$ with $\inf x^*_0(C)>0,$
there exists a family $\{T_i:i\in I\}$, where each $T_i$ is a finite union of closed balls 
 in $X$, such that $C\subset \bigcap\limits_{i\in I} T_i $ 
and $(\bigcap \limits_{i\in I} T_i) \bigcap H =\emptyset.$
\end{enumerate}
Then (i) $\Rightarrow$ (ii) $\Leftrightarrow$ (iii) $\Leftrightarrow$ (iv).
\end{theorem}

\begin{proof}

The equivalence (ii) $\Leftrightarrow$ (iv) follows from arguments analogous to those in Theorem $\ref{w star semi denting char}$ (i) $\Leftrightarrow$ (iii).

(i) $\Rightarrow$ (ii). Let $C$ be a bounded set in $X$ with $\inf x^*_0(C)>0$. As in Theorem $\ref{w star semi pc char}$ (i) $\Rightarrow$ (ii), we may assume without loss of generality that $C\subset B_X$.


 Choose $0<\varepsilon <\inf x^*_0(C)$. Since $x^*_0$ is semi $w^*$-SCS point of $B_{X^*},$ then there exists a convex combination of $w^*$-slices $S=\sum\limits_{i=1}^{n} \lambda_i S_i$ of $B_{X^*}$ such that $S\subset B(x^*_0,\frac{\varepsilon}{2})$. Then proceeding with similar technique as in the proof of (i) $\Rightarrow$ (ii) of Theorem $\ref{ch w star bscsp}$, we get closed balls $B_1,B_2,\ldots,B_n$ in $X$ such that $C\subset \bigcup\limits_{i=1}^{n} B_i$ and $0\notin \bigcup\limits_{i=1}^{n} B_i.$
 
 (iii) $\Rightarrow$ (ii). Straightforward.
 
 (ii) $\Rightarrow$ (iii).
  Let $C$ be bounded in $X^{**}$ such that $\inf x^*_0(C)> 0.$ Thus, $C\subset rB_{X^{**}}$ for some $r>0.$ Choose $0 < \delta < \inf x^*_0(C)$.
   Consider $A=\{ x\in rB_X : x^*_0(x)>\delta\}.$
   Then, proceeding with a similar technique as in proof (ii) $\Rightarrow$ (iii) of Theorem $\ref{ch w star bscsp}$, we obtain closed balls $B_1 , \ldots , B_n$ in $X^{**}$ with centers in $X$ such that 
\begin{center}
$C\subset \overline{A}^{w^*} \subset  \bigcup\limits_{i=1}^{n} B_i$ and $0\notin  \bigcup\limits_{i=1}^{n} B_i.$
\end{center}

\end{proof}

\begin{remark}
The validity of the implication (ii) $\Rightarrow$ (i) of Theorem $\ref{wstarsemiscs}$ is not known in general.
\end{remark}

\begin{corollary}
Let $X$ be a Banach space and every nonzero element in $B_{X^*}$ be semi $w^*$-SCS point of $B_{X^*}.$ Then for every bounded set $C$ of $X$ and any hyperplane $H$ in $X$ with $d(C,H)>0$, there exists a family $\{T_i:i\in I\}$, where each $T_i$ is a finite union of closed balls 
 in $X$, such that $C\subset \bigcap\limits_{i\in I} T_i $ 
and $(\bigcap \limits_{i\in I} T_i) \bigcap H =\emptyset.$
\end{corollary}

\begin{corollary}
Let $X$ be a Banach space, $x_0 \in B_{X}\setminus\{0\}$ 
 and $H=\{x^*\in X^* : x^*(x_0)=0\}$. 
Consider the following statements. 
\begin{enumerate}
\item $x_0$ is a semi SCS point of $B_{X}$.

\item For every bounded set $C$ of $X^*$ with $\inf x_0(C)>0,$ then there exist closed balls $B_1,B_2,\ldots,B_n$ in $X^*$ such that $C\subset \bigcup\limits_{i=1}^{n} B_i$ and $0\notin \bigcup\limits_{i=1}^{n} B_i.$
\item For every bounded set $C$ of $X^*$ with $\inf x_0(C)>0,$
 there exists a family $\{T_i:i\in I\}$, where each $T_i$ is a finite union of closed balls 
 in $X^*$, such that $C\subset \bigcap\limits_{i\in I} T_i $ 
and $(\bigcap \limits_{i\in I} T_i) \bigcap H =\emptyset.$
\end{enumerate}
Then (i) $\Rightarrow$ (ii) $\Leftrightarrow$ (iii).
\end{corollary}

\begin{proof}
By the arguments used in the proof of Proposition $\ref{A1}$, it follows that
$x_0\in B_X$ is a semi SCS point of $B_X$ if and only if $x_0$ is semi $w^*$-SCS point of $B_{X^{**}}$.
The result follows from Theorem $\ref{wstarsemiscs}$.
\end{proof}



\begin{corollary}
Let $X$ be a Banach space and every nonzero element in $B_{X}$ be semi SCS point of $B_{X}.$ Then for every bounded set $C$ of $X^*$ and any $w^*$-closed hyperplane $H$ in $X^*$ with $d(C,H)>0$, there exists a family $\{T_i:i\in I\}$, where each $T_i$ is a finite union of closed balls 
 in $X^*$, such that $C\subset \bigcap\limits_{i\in I} T_i $ 
and $(\bigcap \limits_{i\in I} T_i) \bigcap H =\emptyset.$
\end{corollary}

\begin{proposition}\label{scs oneside}
Let $X$ be a Banach space. If $w^*$-SCS points of $B_{X^*}$ are dense in $B_{X^*}$, then all points of $B_{X^*}$ are semi $w^*$-SCS points of $B_{X^*}$.
\end{proposition} 

\begin{proof}
Let $x^*\in B_{X^*}$ and $\varepsilon >0.$ Then there exists a $w^*$-SCS point $x^*_0\in B_{X^*}$ such that $\|x^*_0-x^*\|<\frac{\varepsilon}{2}.$ Since $x^*_0$ is $w^*$-SCS point of $B_{X^*}$, then there exists a convex combination of $w^*$-slices $T,$ containing $x^*_0,$ in $B_{X^*}$ such that diameter of $T$ is less than $\frac{\varepsilon}{2}$.
 Thus, $T\subset B(x^*,\varepsilon).$ Hence, $x^*$ is a semi $w^*$-SCS point of $B_{X^*}.$
\end{proof}

\begin{remark}

It is known from \cite{CL2} and \cite{GGS} that all points of $S_{X^*}$ are semi $w^*$-denting points if and only if the $w^*$-denting points of $B_{X^*}$ are dense in $S_{X^*}$. Similarly, Theorem $\ref{ch propii}$ and Theorem $\ref{propii}$ yield an analogous equivalence for semi $w^*$-PC and $w^*$-PC of $B_{X^*}$. However, for a similar connection between semi $w^*$-SCS points and $w^*$-SCS points, one implication is true from Proposition $\ref{scs oneside}$, while the converse remains unknown. More precisely, it is not known whether the fact that all points of $B_{X^*}$ are semi $w^*$-SCS points of $B_{X^*}$ implies that the set of $w^*$-SCS points of $B_{X^*}$ is dense in $B_{X^*}$.

\end{remark}

We present a summary of the interrelationships among the notions discussed in this section. Additionally, we provide counterexamples to show that none of the 
reverse implications hold. 



\begin{center}
semi denting \hspace{.3 cm} $\Longrightarrow$ \hspace{.2 cm} semi PC \hspace{.3 cm} $\Longrightarrow$ \hspace{.2 cm} semi SCS \\
 $  \Big \Uparrow \hspace{3 cm} \Big \Uparrow \hspace{3 cm} \Big \Uparrow$  \\
semi $w^*$-denting $\Longrightarrow$  semi $w^*$-PC $\Longrightarrow$  semi $w^*$-SCS
\end{center}
\begin{example}\label{example}
~

\begin{itemize}


\item Consider the space $l_1^n$. Then $w^*$-PC($B_{l_1^n}$) = PC($B_{l_1^n}$) = $B_{l_1^n}$ and by Remark $\ref{dent-ext}$, 
\begin{center}
semi denting($B_{l_1^n}$) = semi $w^*$-denting($B_{l_1^n}$) = $\{\pm e_i : 1\leqslant i \leqslant n\}.$
\end{center}
 Hence, any point in $B_{l_1^n}$ other than  $\pm e_i$ for $1\leqslant i \leqslant n$ is a semi PC (semi $w^*$-PC) of $B_{l_1^n}$, but not a semi denting point (semi $w^*$-denting point) of $B_{l_1^n}$.

\item Observe that in any Banach space $X$, if $B_X$ contains a semi SCS (resp. semi $w^*$-SCS) point, then $0$ is necessarily a semi SCS (resp. semi $w^*$-SCS) point of $B_X$. For instance, in $B_{l_1}$, the point 0 is both semi SCS and semi $w^*$-SCS point of $B_{l_1}$.  However, 0 can not be semi PC or semi $w^*$-PC of $B_{l_1}$. 

\item  Consider $X^*=C[0,1]^*.$ It is known that $X^*$ has the $w^*$-strong diameter two property, that is, every convex combination of $w^*$-slices of $B_{X^*}$ has diameter two (see \cite{BGLPRZ}, \cite{HLP1}).
Consequently, $B_{X^*}$ fails to contain any semi $w^*$-SCS point (and hence no semi $w^*$-PC or semi $w^*$-denting point). However, $B_{X^*}$ has denting points (for example, the characteristic function $\chi_{(0,1]}$), and hence it also contains semi denting point, semi PC and semi SCS point.

\end{itemize}

\end{example}


\section{$\mathcal{A}$-SCS point and its ball separation characterizaton}
\hypertarget{6.2}{ }

In \cite{CL1}, the concepts of $\mathcal{A}$-denting point and $\mathcal{A}$-Point of Continuity ($\mathcal{A}$-PC) of $B_{X^*}$ were introduced and their ball separation characterization was studied.
Motivated by the concepts of $\mathcal{A}$-denting point and $\mathcal{A}$-PC of $B_{X^*}$,
 we define,

\begin{definition}
Let $\mathcal{A}$ be a collection of bounded subsets of $X.$ Then $x^*\in X^*$ is said to be an \emph{$\mathcal{A}$-Small Combination of Slice point} ($\mathcal{A}$-SCS point)    of $B_{X^*}$ if for each $A\in \mathcal{A}$ and $\varepsilon>0$ there exists a convex combination of $w^*$-slices $T=\sum\limits_{i=1}^{n} \la_i S_i$ in $B_{X^*}$ 
 such that $x^* \in T$ and diam$_A(T) <\e$.
\end{definition} 

\begin{remark}
If $\mathcal{A}$, in particular, contains all bounded subsets of $X$, then $\mathcal{A}$-SCS point coincides with $w^*$-SCS point of $B_{X^*}$. 
  If $\mathcal{A}$ is taken to be the family of all compact subsets of $X$, then an  $\mathcal{A}$-denting point of $B_{X^*}$ is exactly an extreme point of $B_{X^*}$ \cite{WZ} and every point on $S_{X^*}$ is an $\mathcal{A}$-PC of $B_{X^*}$ \cite{BS2}. Consequently, every point of $B_{X^*}$ becomes an $\mathcal{A}$-SCS point of $B_{X^*}$. See also \cite{BG} and \cite{DC}, where related concepts were investigated for other choices of $\mathcal{A}$.
\end{remark}

\begin{definition}
\cite{CL1}
A collection $\mathcal{A}$ of bounded subsets of $X$ is said to be \emph{compatible}  if it satisfies the followings :
\begin{enumerate}
\item If $A\in \mathcal{A}$, then for each $C\subset A ,$ we have $C\in \mathcal{A}.$
\item If $A\in \mathcal{A}$ and $x\in X,$ then $A+x \in \mathcal{A}$ and $A\bigcup \{x\} \in \mathcal{A}.$
\item If $A\in \mathcal{A},$ then $\overline{|\mathrm{co}|}(A)\in \mathcal{A}.$ 
\end{enumerate}
\end{definition}
 


In our next result we discuss ball separation conditions for nonzero $\mathcal{A}$-SCS point of $B_{X^*}$.


\begin{theorem}\label{scs prop}
Let $X$ be a Banach space, $\mathcal{A}$ be a collection of bounded subsets of $X$
and $x^*_0$ be an $\mathcal{A}$-SCS point of $B_{X^*}\setminus\{0\}.$ Then for every $A\in \mathcal{A}$ with $\inf x^*_0(A)>0,$  there exist closed balls $B_1,B_2,\ldots,B_n$ in $X$ such that $A\subset \bigcup\limits_{i=1}^{n} B_i$ and $0\notin \bigcup\limits_{i=1}^{n} B_i.$
\end{theorem}
\begin{proof}
Suppose $x^*_0$ is an $\mathcal{A}$-SCS point of $B_{X^*}\setminus\{0\},$ $A\in \mathcal{A}$ and $\inf x^*_0(A)=\delta >0.$ Then there exists convex combination of $w^*$-slices
$\sum\limits_{i=1}^{n} \lambda_i S_i$ of $B_{X^*}$ 
 such that $x^*_0\in \sum\limits_{i=1}^{n} \lambda_i S_i$ and diam$_{A} (\sum\limits_{i=1}^{n} \lambda_i S_i)<\frac{\delta}{2}.$ 
 Observe that for any $a\in A$ and $x^*\in \sum\limits_{i=1}^{n} \lambda_i S_i,$ we have
$$x^*(a)
= x^*_0(a)-( x^*_0 - x^* )(a)
\geqslant \delta - \|x^*_0-x^*\|_A 
\geqslant \delta - \frac{\delta}{2}
= \frac{\delta}{2} .$$
Thus $\inf x^* (A) \geqslant \frac{\delta}{2} >0$ for all $x^*\in \sum_{i=1}^{n} \lambda_i S_i.$ \\
We claim that   $A=\bigcup\limits_{i=1}^{n} A_i$ where $A_i= \{ x\in A : x^*(x)> 0 \ \forall x^*\in S_i\}.$\\
 If not, there exists $x\in A \setminus \bigcup\limits_{i=1}^{n}A_i.$ Then for each i, there exist $x_i^* \in S_i$ such that $x_i^*(x)\leqslant 0.$ Hence, $\sum\limits_{i=1}^{n} \lambda_i x_i^* (x)\leqslant 0,$ a contradiction, since $\inf x^* (A) >0$ for all $ x^* \in \sum\limits_{i=1}^{n} \lambda_i S_i.$
 
Consider $B_1(A_i)=\{x^*\in B_{X^*} : x^*(x)\leqslant 0 $ for some $x\in A_i \}.$ Then $B_1(A_i) \bigcap S_i =\emptyset.$  
 By Proposition $\ref{CLcoro result 1}$, for each i, there exist closed balls $B_i$  in $X$ such that 
$A_i \subset B_i,$   $  0\notin B_i.$  
Thus, $A \subset \bigcup\limits_{i=1}^{n} B_i$ and $ 0\notin \bigcup\limits_{i=1}^{n} B_i.$ 
\end{proof}

\begin{remark}
The converse of Proposition $\ref{scs prop}$ is not known in general. However, a complete ball separation characterization for $\mathcal{A}$-denting point or $\mathcal{A}$-PC is available only when $\mathcal{A}$ is a compatible collection \cite{CL1}. This naturally raises the question of whether the converse of Proposition $\ref{scs prop}$ holds when $\mathcal{A}$ is a compatible collection.
\end{remark}

Let $\mathcal{A}$ be a compatible collection of bounded subsets of $X$. We define $\tau_\mathcal{A}$ as the topology on $X^*$ generated by the seminorms $\{\|.\|_A : A\in \mathcal{A}\}.$ The collection $\{B_A(f, \varepsilon) : A\in \mathcal{A}, \varepsilon >0\}$ forms a local base for $\tau_\mathcal{A}$ at 0. 
\begin{definition}
\cite{CL3}, \cite{GK} A subset $A$ of a Banach space $X$ is said to be \emph{ball generated} if there is a family $\{T_i : i \in I\}$ such that each $T_i$ is a finite union of closed balls and $A=\bigcap\limits_{i\in I} T_i$. A Banach space $X$ is said to have the \emph{Ball Generated Property} (BGP) if every closed bounded convex
set in $X$ is ball generated.
\end{definition}

\begin{theorem}\label{improve}
Suppose $X$ is a Banach space and $\mathcal{A}$ is a compatible collection of bounded subsets of $X.$ If the linear span of $\mathcal{A}$-SCS points of $B_{X^*}$ is $\tau_{\mathcal{A}}$-dense in $X^*,$ then every closed bounded convex subset in $\mathcal{A}$ is ball generated.
\end{theorem}
\begin{proof}
Let $A$ be a closed bounded convex set in $\mathcal{A}$.
Choose $x_0\notin A.$ 
Without loss of generality we can assume that $x_0=0.$ Otherwise, we will consider $A-x_0$, which is in $\mathcal{A}$, since $\mathcal{A}$ is compatible. If $A-x_0$ is ball generated, then $A$ is also ball generated.

 Then by Hahn-Banach separation theorem there exists $x^*\in S_{X^*}$ such that $\inf x^*(A)>0.$ Let $\inf x^*(A)=\delta>0.$ Since the linear span of $\mathcal{A}$-SCS points of $B_{X^*}$ is $\tau_{\mathcal{A}}$-dense in $X^*,$ then there are scalars $\lambda_1,\ldots,\lambda_n$  and $\mathcal{A}$-SCS points $x^*_1,\ldots,x^*_n$ of $B_{X^*}$ such that 
\begin{equation}
\|x^*-\sum_{i=1}^{n} \lambda_i x^*_i\|_A<\frac{\delta}{2}.
\end{equation}
 Furthermore, since the set of  $\mathcal{A}$-SCS points of $B_{X^*}$ is symmetric, we may assume that $\lambda_i>0$ for all $i=1,\ldots,n.$ 
  Observe that, for any $a\in A,$
$$\sum_{i=1}^{n} \lambda_i x^*_i (a)
= x^*(a)+\Big( -x^* + \sum_{i=1}^{n} \lambda_i x^*_i \Big )(a)
\geqslant \delta - \|x^*-\sum_{i=1}^{n} \lambda_i x^*_i\|_A 
\geqslant  \frac{\delta}{2} .$$
Thus $\inf (\sum\limits_{i=1}^{n} \lambda_i x^*_i) (A) \geqslant \frac{\delta}{2}.$ Then  $\sum\limits_{i=1}^{n} \lambda_i x^*_i \neq 0$ and therefore, without loss of generality we can assume that $x_i^*\neq 0$ for all $i=1\ldots, n$.\\ 
We claim that $A=\bigcup\limits_{i=1}^{n} A_i$ where $A_i= \{ x\in A : x^*_i(x) \geqslant \frac{\delta}{2 n \lambda_i} \}.$\\
 If not, there exists $x\in A \setminus \bigcup\limits_{i=1}^{n} A_i.$ Then  $x^*_i(x)< \frac{\delta}{2 n \lambda_i}$ for all i. Hence, $\sum\limits_{i=1}^{n} \lambda_i x^*_i (x) < \frac{\delta}{2 },$ a contradiction, since $\inf \sum\limits_{i=1}^{n} \lambda_i x^*_i (A) \geqslant \frac{\delta}{2}.$
 
  For each i, $A_i\in \mathcal{A}$ with $\inf x^*_i(A_i)>0,$ then by Theorem $\ref{scs prop}$ there exist closed balls $B_{i,1},B_{i,2},\ldots,B_{i,m(i)}$ in $X$ such that 
  $$A_i\subset \bigcup_{j=1}^{m(i)} B_{i,j} \ \ \mathrm{and} \ \ 0\notin \bigcup_{j=1}^{m(i)} B_{i,j}.$$ 
  Thus, $A \subset \bigcup\limits_{i=1}^{n}\bigcup\limits_{j=1}^{m(i)} B_{i,j}$ and $x_0=0\notin \bigcup\limits_{i=1}^{n}\bigcup\limits_{j=1}^{m(i)} B_{i,j}.$ 
  
  Therefore, for each $x\notin A$, there exists finite union of closed balls $T_x$ in $X$ such that $A\subset T_x$ and $x\notin T_x.$ 
  Hence, $A=\bigcap\limits_{x\notin A} T_x$.
\end{proof}

\begin{remark}
In particular, considering
 $\mathcal{A}$ as the set of all bounded subsets of $X$, we get from Theorem $\ref{improve}$ 
that $X$ has the BGP
if $X^*$ is the closed linear span of $w^*$-SCS points of $B_{X^*}$, which was proved in \cite[Proposition 2.5]{BB}. Thus Theorem $\ref{improve}$ provides a generalization of \cite[Proposition 2.5]{BB}.
\end{remark}

\begin{Acknowledgement} 
	The first author is grateful to Professor Ethan Duckworth, Chair, Department of mathematics and statistics  for his support and encouragement. She  is also grateful to Professor Bahram Roughani, Associate Dean College of Arts and Sciences, Loyola University for providing her with Dean's supplemental grant during her travel in Summer 2024. The research of the second author is financially supported by the Institute Postdoctoral Fellowship, NISER Bhubaneswar.
\end{Acknowledgement}

\end{document}